\documentclass[english]{smfart}
\usepackage{latexsym}
\usepackage{amsfonts}
\usepackage{amssymb}
\usepackage{amsmath}
\usepackage{amscd}
\usepackage{amsthm}
\usepackage[bookmarks=false]{hyperref}
\numberwithin{equation}{section}
\newtheorem*{mainresult}{Corollary \ref{Cbiggergroups}}
\author{Andrew Obus}
\address{Columbia University \\ Department of Mathematics \\ MC4403 \\ 2990 Broadway \\ New York, NY 10027}
\email{obus@math.columbia.edu}
\title{Toward Abhyankar's Inertia Conjecture For $PSL_2(\ell)$}

\def\ints{{\mathbb Z}}
\def\rats{{\mathbb Q}}
\def\nats{{\mathbb N}}
\def\reals{{\mathbb R}}
\def\proj{{\mathbb P}}

\def\FF{{\mathbb F}}
\def\qed{\hfill $\Box$}

\def\Frac{{\text{Frac}}}
\def\Spec{{\text{Spec }}}

\DeclareMathOperator{\ord}{ord}

\def\mc#1{\mathcal{#1}}
\def\ol#1{\overline{#1}}

\newtheorem{theorem}{Theorem}[section]
\newtheorem{prop}[theorem]{Proposition}
\newtheorem{lemma}[theorem]{Lemma}
\newtheorem{corollary}[theorem]{Corollary}
\newtheorem{preremark}[theorem]{Remark}
\newenvironment{remark}{\begin{preremark}\rm}{\end{preremark}}

\begin{document}
\frontmatter
\begin{abstract} 
For $l \ne p$ odd primes, we examine $PSL_2(\ell)$-covers of the projective line branched at one point over an algebraically closed field $k$ of characteristic $p$, 
where $PSL_2(\ell)$ has order divisible by $p$.  We show that such covers can be realized with a large variety of inertia groups.
Furthermore, for each inertia group realized, we can realize all ``sufficiently large" higher ramification filtrations.
\end{abstract}
\thanks{The author was supported by an NSF Postdoctoral 
Research Fellowship in the Mathematical Sciences}

\maketitle

\mainmatter
\section{Introduction}\label{CHintro}

Over an algebraically closed field $k$ of characteristic $0$, finite algebraic branched covers $Y \to \proj^1$ with $n$ fixed branch points are in one-to-one 
correspondence with finite topological branched covers of the Riemann sphere with $n$ fixed branch points.  Both correspond to finite index
subgroups of $\pi_1(\proj^1_k \backslash \{x_1, \ldots, x_n\})$, the free (profinite) group on $n-1$ generators.  In particular, there exist no nontrivial 
covers of $\proj^1_k$ branched at one point (hereafter called ``one-point covers").  

If $k$ is algebraically closed of characteristic $p$, the situation differs in two important ways.  First, there exist many covers which do not have topological
analogs.  In particular, as a consequence of \emph{Abhyankar's conjecture} (proven by Raynaud (\cite{Ra:ab}) in the case of the affine line and Harbater 
(\cite{Ha:ac}) in general), it follows that a finite
group $G$ can be realized as the Galois group of an $n$-point cover of $\proj^1_k$ exactly when $G/p(G)$ can be generated by $n-1$ elements (here,
$p(G)$ is the subgroup of $G$ generated by all of the $p$-Sylow subgroups).  Thus, $G$ occurs as the Galois group of a one-point cover iff $G = p(G)$.
Such a group is called \emph{quasi-$p$}.  The second major difference is that, unlike in characteristic zero, the genus of $Y$ is not determined by the 
degree, ramification points, and ramification indices of $f:Y \to \proj^1$, but also depends on wild ramification behavior at the ramification points.  This 
behavior is encoded in the higher ramification filtration (\S\ref{Sram}).

While Abhyankar's conjecture guarantees existence of certain one-point $G$-covers, it does not provide examples.  In particular, the question of
what subgroups $I \subseteq G$ 
can occur as the inertia group of a $G$-Galois one-point cover has been answered in only a few cases.  By basic ramification theory, 
such a group $I$ must be of the form $P \rtimes \ints/m$, where $P$ is a $p$-group and $p \nmid m$.  Furthermore, $I$ must also generate $G$ as a 
normal subgroup (this is automatic, for instance, when $G$ is simple).
\emph{Abhyankar's inertia conjecture} (\cite{Ab:rs}) states that any $I$ satisfying these properties occurs as the inertia group of a one-point 
$G$-cover.  This is known to be true for $G = PSL_2(p)$ and $G = A_p$, for $p \geq 5$ (\cite{BP:rr}), and for $A_{p+2}$ when 
$p \equiv 2 \pmod{3}$ is an odd prime (\cite{MP:ac}).  It is clearly true in some trivial cases, for instance when $G$ is abelian.
However, outside of these cases, our knowledge is limited.   Abhyankar has constructed many examples of one-point $G$-covers with various inertia 
groups where $G$ is simple,
but all of these examples are in cases where $G$ is alternating, sporadic, or of Lie type over a field of characteristic $p$.  
In particular, if $G$ is a simple group with a cyclic $p$-Sylow subgroup of order
greater than $p$ (such a group cannot be alternating, sporadic, or of Lie type over a field of characteristic $p$), 
then as far as I know, no examples of one-point $G$-covers have been constructed, let alone with
any particular inertia group.  

In this paper, we investigate Abhyankar's inertia conjecture for certain simple groups $G = PSL_2(\ell)$, where $p \, | \, |G|$ and $\ell \ne p$ are odd 
primes.  In this case, the group $PSL_2(\ell)$ has a cyclic $p$-Sylow subgroup of order $p^a$, where $a = v_p(\ell^2-1)$.
While we are not able to prove the conjecture in full (see Remark \ref{Rabhyankar}), we exhibit many examples with a wide variety of inertia groups:

\begin{mainresult}
Let $k$ be an algebraically closed field of characteristic $p \geq 7$ and $G \cong PSL_2(\ell)$, where $p \, | \, |G|$. 
Suppose $I$ is either a cyclic group of order $p^r$ or a dihedral group $D_{p^r}$ of order $2p^r$, with $1 \leq r \leq v_p(|G|)$.  Then there exists a
$G$-cover $f: Y \to \proj^1_k$ branched at one point with inertia groups isomorphic to $I$.
\end{mainresult}

This result generalizes \cite[Theorem 3.6]{BP:rr}, which deals with the case $a=1$. 
We also investigate the higher ramification behavior of one-point $PSL_2(\ell)$-covers with inertia group $I$.  We show that, in the situation of
Corollary \ref{Cbiggergroups}, any sufficiently ``large" higher ramification filtration must occur.  See Corollary \ref{Chigherram} for a more specific result.

\section*{Acknowledgments}
I thank Rachel Pries for many useful conversations on this topic.

\section{Preliminaries}\label{Sprelims}
Notation:  If $G$ is a group with cyclic $p$-Sylow subgroup $P$, we write $m_G$ for $|N_G(P)/Z_G(P)|$, the normalizer modulo the centralizer.
If $R$ is a local ring or a discretely valued field, then $\hat{R}$ is the usual completion.  If $x$ is a point of a scheme $X$, then $\mc{O}_{X, x}$ is
the local ring of $X$ at $x$.

\subsection{Higher ramification filtrations}\label{Sram}
We recall some facts from \cite[IV]{Se:lf}.
Let $K$ be a complete discrete valuation field with algebraically closed 
residue field $k$ of characteristic $p > 0$.  If $L/K$ is a finite Galois
extension of fields with Galois group $G$, then 
$L$ is also a complete discrete valuation field with residue field $k$.  Here
$G$ is of the form $P \rtimes \ints/m$, where
$P$ is a $p$-group and $m$ is prime to $p$.  The group $G$ has a filtration $G 
\supseteq G^i$ ($i \in \reals_{\geq 0}$) called the \emph{higher ramification filtration for the upper numbering}.  
If $i \leq j$, then $G^i \supseteq G^j$ (see \cite[IV, \S1, \S3]{Se:lf}).   
The subgroup $G^i$ is known as the \emph{$i$th higher
ramification group for the upper numbering}.  One knows that $G^0 = G$, and that for sufficiently small
$\epsilon > 0$, $G^{\epsilon} = P$.  For sufficiently large $i$, $G^i = \{id\}$. Any $i$ such that $G^i \supsetneq G^{i + \epsilon}$ 
for all $\epsilon > 0$ is called an \emph{upper jump} of the extension $L/K$.  If $i$ is an upper jump and $i > 0$,
then $G^i/G^{i + \epsilon}$ is an elementary abelian $p$-group for sufficiently small $\epsilon$.  
The greatest upper jump (i.e., the greatest $i$ such that $G^i \neq \{id\}$) is
called the \emph{conductor} of higher ramification of $L/K$.  

If $P \cong \ints/p^r$ is cyclic, then $G$ must have $r$ different positive upper jumps $u_1 < \cdots < u_r$.  Since the sequence 
$(u_1, \ldots, u_r)$ encodes the entire higher ramification filtration, we will simply say that such an extension (or the inertia group of
such an extension) has upper higher ramification filtration $(u_1, \ldots, u_r)$ in this case.

The higher ramification filtration is important because if $f:Y \to X$ is a branched cover of curves in characteristic $p$, 
and $y \in Y$ is a ramification point and $f(y) = x$, then the higher ramification filtration for $\hat{\mc{O}}_{Y,y}/\hat{\mc{O}}_{X,x}$ 
figures into the ramification divisor term in the Hurwitz formula.  Specifically, let $f: Y \to \proj^1$ be a one-point $G$-cover with inertia
groups $I \cong \ints/p^r \rtimes \ints/m$ with $p \nmid m$ and higher ramification filtration $(u_1, \ldots, u_r)$.  
Then, by \cite[Lemma 1]{Pr:lg}, the genus of $Y$ is $1 - |G| + \frac{|G|\deg(R)}{2mp^r}$, where $R$ is the ramification
divisor, which has degree $$mp^r - 1 + (p-1)m(u_1 + pu_2 + \cdots + p^{r-1}u_r).$$

\subsection{Stable reduction}\label{Sstable}
Let $X \cong \proj^1_K$, where $K$ is a characteristic zero complete discretely
valued field with 
algebraically closed residue field $k$ of characteristic $p>0$ (e.g., $K$ is the completion
of the maximal unramified extension of $\rats_p$).  Let $R$ be the valuation ring of $K$.  

Let $f: Y \to X$ be a $G$-Galois cover defined over $K$, with $G$ any finite
group, such that the branch points of $f$ are defined over $K$ and their
specializations 
do not collide on the special fiber of $X_R$.  Assume that
$2g_X - 2 + r \geq 1$, where $g_X$ is the genus of $X$ and $r$ is the number of branch points of $f$.  By a
theorem of Deligne and Mumford (\cite[Corollary 2.7]{DM:ir}), combined with work
of Raynaud (\cite{Ra:sp}) and Liu (\cite{Li:sr}), there is a
minimal finite extension $K^{st}/K$ 
with ring of integers $R^{st}$, and a unique model $Y^{st}$ of $Y_{K^{st}}$
(called the \emph{stable model}) such that

\begin{itemize}
\item The special fiber $\ol{Y}$ of $Y^{st}$ is semistable (i.e., it is reduced,
and has only nodes for singularities). 
\item The ramification points of $f_{K^{st}} = f \times_K K^{st}$ specialize to
\emph{distinct} smooth points of $\ol{Y}$.
\item Any genus zero irreducible component of $\ol{Y}$ contains at least three
marked points (i.e., ramification points or points of intersection with the rest
of $\ol{Y}$).
\end{itemize}
Since the stable model is unique, it is acted upon by $G$, and we set $X^{st} =
Y^{st}/G$.  Then $X^{st}$ is semistable (\cite{Ra:pg}) and can be naturally identified with a blowup of $X \times_R R^{st}$
centered at closed points.  The map $f^{st}: Y^{st} \to X^{st}$ is called the \emph{stable model} of $f$. 
The special fiber $\ol{f}: \ol{Y} \to \ol{X}$ of $f^{st}$ is called the \emph{stable reduction} of $f$.  

From now on, assume that $f$ has bad reduction (i.e., $\ol{X}$ is reducible).
Any irreducible component $\ol{X}_b$ of $\ol{X}$ that intersects the rest of $\ol{X}$ in only one point is called a $\emph{tail}$.
If $G$ acts without inertia above the generic point of $\ol{X}_b$, we call $\ol{X}_b$ an \emph{\'{e}tale tail}.  If, furthermore, 
$\ol{X}_b$ contains the specialization of a branch point of $f$, then $\ol{X}_b$ is called \emph{primitive}.
If not, $\ol{X}_b$ is called \emph{new}.  On an \'{e}tale tail $\ol{X}_b$, branching of $\ol{f}$ can only take place at the point where
$\ol{X}_b$ intersects the rest of $\ol{X}$ and, if $\ol{X}_b$ is primitive, where the branch point of $f$ specializes.

Suppose $G$ has a cyclic $p$-Sylow subgroup $P$.
Consider an \'{e}tale tail $\ol{X}_b$ of $\ol{X}$.  Let $\ol{x}_b$ be the unique point at which $\ol{X}_b$ intersects the rest of $\ol{X}$.
Let $\ol{Y}_b$ be a component of $\ol{Y}$ lying above $\ol{X}_b$, and let $\ol{y}_b$ be a point
lying above $\ol{x}_b$.  Then the \emph{effective ramification invariant} $\sigma_b$ is the conductor of higher ramification of 
the extension $\Frac(\hat{\mc{O}}_{\ol{Y}_b, \ol{y}_b})/\Frac(\hat{\mc{O}}_{\ol{X}_b, \ol{x}_b})$.  

The most important result that we will use is a special case of the \emph{vanishing cycles formula}.  This formula relates the stable reduction
of $f$ to the genus of $X$ and the branching behavior of $f$.  In particular, it places a limit on how many tails $\ol{X}$ can have and how large their
effective ramification invariants can be.  If $|G|$ is prime to $p$, then $f$ has good reduction, and there is no need for such a formula.  
The original version of the vanishing cycles formula, proved by Raynaud (\cite[\S3.4.2 (5)]{Ra:sp}), requires that $G$ has a $p$-Sylow subgroup of 
order $p$.  A generalized version was proven by the author in \cite[Theorem 3.14]{Ob:vc}, which applies when $G$ has a cyclic $p$-Sylow subgroup of
arbitrary order (Theorem \ref{Tvancycles} below is a special case, corresponding to \cite[Corollary 3.15]{Ob:vc}).  The formula will be essential
for us to exhibit a one-point cover in characteristic $p$ whose Galois group has large cyclic $p$-Sylow group but whose inertia groups are small. 

\begin{theorem}[Vanishing cycles formula]\label{Tvancycles}
Let $f: Y \to X \cong \proj^1_K$ be a three-point $G$-Galois cover with bad reduction, where $G$ has a cyclic
$p$-Sylow subgroup.  Let $B_{\text{new}}$ be an indexing set for the new \'{e}tale tails and let
$B_{\text{prim}}$ be an indexing set for the primitive \'{e}tale tails.  Then  

\begin{equation}\label{Evancycles} 
1 = \sum_{b \in B_{\text{new}}} (\sigma_b - 1) + \sum_{b \in B_{\text{prim}}} \sigma_b.
\end{equation}
\end{theorem}

\begin{remark}\label{Rpositiveterms}
Note that, by \cite[Lemma 4.2 (i)]{Ob:vc}, each term on the right hand side of (\ref{Evancycles}) is at least $\frac{1}
{m_G}$.  In particular, each term is positive.
\end{remark}

\subsection{The auxiliary cover}\label{Saux}
Retain the assumptions from the beginning of \S\ref{Sstable} (in particular, $G$ need not have a cyclic $p$-Sylow
subgroup).  Assume that $f: Y \to X$ is a $G$-cover 
defined over $K$ as in \S\ref{Sstable} with bad reduction, so that $\ol{X}$ is not just the original component.
By \cite[\S2.6]{Ob:fm2} (see also \cite[\S3.2]{Ra:sp}), we can construct an \emph{auxiliary cover} $f^{aux}: Y^{aux} \to X$ over some finite extension $K'$ of $K$.  
The cover $f^{aux}$ is a (connected) $G^{aux}$-cover, for some subgroup $G^{aux} \subseteq G$.  

The only properties we will need of $f^{aux}$ are the following: 
\begin{prop}\label{Pauxgroups}
\begin{description}
\item{(i)} For each branch point of $f$ of (branching) index $e$ with $v_p(e) \geq 1$, the cover $f^{aux}$ has a branch point of 
index $e'$ with $v_p(e') = v_p(e)$.
For each \'{e}tale tail of $\ol{X}$, the cover $f^{aux}$ may have a branch point of prime-to-$p$ index.  There are no other branch points of $f^{aux}$.  
\item{(ii)} If a $p$-Sylow subgroup of $G$ is \emph{cyclic}, then the group
$G^{aux}$ has a normal subgroup $N$ of prime-to-$p$ order such that $G/N \cong \ints/p^r \rtimes \ints/m$, with $p \nmid m$.  
\end{description}
\end{prop}

\begin{proof}
Part (i) follows from the construction in \cite[\S2.6]{Ob:fm2}.
Part (ii) follows from \cite[Proposition 2.12]{Ob:fm2} and \cite[Proposition 2.4(i)]{Ob:vc}.
\end{proof}

\section{A one-point cover}\label{Sonepoint}
We maintain the notations $K$, $R$, and $k$ of \S\ref{Sprelims}.  We write $\ol{K}$ for the algebraic closure of $K$.

\begin{lemma}\label{Lexistence}
Let $\ell \ne p$ be odd primes such that $v_p(\ell^2 - 1) =: a \geq 2$.  If $G \cong PSL_2(\ell)$, then there exists a three-point $G$-cover $f:Y \to X = 
\proj^1_{\ol{K}}$ whose branch points have branching indices $e_1$, $e_2$, and $e_3$, with $0 = v_p(e_1) < v_p(e_2) < v_p(e_3)$.
\end{lemma}

\begin{proof}
If $H \cong SL_2(\ell)$, it suffices to exhibit a three-point $H$-cover with branching indices $e_1'$, $e_2'$, and $e_3'$
satisfying $0 = v_p(e_1') < v_p(e_2') < v_p(e_3')$.  We obtain the desired $G$-cover by quotienting out by $\{ \pm 1 \}$ (as for $i \in \{1, 2, 3\}$, we
will have $e_i = e_i'$ or $e_i = e_i'/2$, depending on the action of $\{ \pm 1 \}$ above the relevant branch point). 
As in \cite[\S3]{BW:mc}, pick an $(\ell -1)$st root of unity $\zeta \in \FF_{\ell}$ and an $(\ell + 1)$st root of unity $\tilde{\zeta} \in \FF_{\ell^2}$.  
Then for $0 < i < \frac{\ell-1}{2}$ (resp.\ $0 < i < \frac{\ell+1}{2}$),
we write $C(i)$ (resp.\ $\tilde{C}(i)$) for the unique conjugacy class of elements of $H$ with eigenvalues $\zeta^{\pm i}$ (resp.\ $\tilde{\zeta}^{\pm i}$).  In 
particular, if $x \in C(i)$ (resp.\ $\tilde{C}(i)$), then $v_p(\ord(x)) = \max(0, v_p((\ell-1)/i))$ (resp.\ $\max(0, v_p((\ell+1)/i))$). 

We construct a triple of conjugacy classes $C = (C_1, C_2, C_3)$ giving rise to the desired cover.  If $v_p(\ell -1) = a$, then pick $C = (\tilde{C}(\frac{\ell+1}{2} - 1), C(\frac{\ell-1}{2} - p), 
C(\frac{\ell-1}{2} - 1))$.  If $v_p(\ell + 1) = a$, then pick $C = (C(\frac{\ell-1}{2} - 1), \tilde{C}(\frac{\ell+1}{2} - p), 
\tilde{C}(\frac{\ell+1}{2} - 1))$.  Note that since $p^2$ divides either $\ell-1$ or $\ell+1$, and $p$ and $l$ are odd, then $2p^2 \leq \ell+1$.  In particular, 
one easily checks $2p + 4 < \ell-1$.  
Now, \cite[Proposition 5.6]{BW:mc} shows that there exist two isomorphism classes of three-point $H$-covers with inertia groups over the three branch 
points lying in $C_1$, $C_2$, and $C_3$ (our $\ell$ is called $p$ in \cite{BW:mc}).  Checking the ramification indices yields
$v_p(e_1') = 0$, $v_p(e_2') = a - 1$, and $v_p(e_3') = a$.
\end{proof}

\begin{lemma}\label{Lnewtail}
The stable reduction of the $G$-cover $f$ in Lemma \ref{Lexistence} has one new \'{e}tale tail and one primitive \'{e}tale tail. 
The new tail $\ol{X}_b$ has effective ramification invariant $\sigma_b = \frac{3}{2}$.
\end{lemma}

\begin{proof}
By \cite[Proposition 2.15]{Ob:vc} and \cite[Proposition 2.4.8]{Ra:sp}, there exists exactly one primitive \'{e}tale tail of $\ol{X}$.  
We note that $m_G = 2$, so Remark \ref{Rpositiveterms} shows that the effective ramification invariant of a new \'{e}tale tail must be at least $\frac{3}{2}$, whereas the invariant of a 
primitive \'{e}tale tail must be at least $\frac{1}{2}$.
If we can show that there exists a new \'{e}tale tail, then the vanishing cycles formula (\ref{Evancycles}) shows that
there is only one, and it has invariant $\frac{3}{2}$.  So we need only show that there exists a new \'{e}tale tail.

Consider the auxiliary cover $f^{aux}: Y^{aux} \to X$ of $f$.  Assume, for a contradiction, that the primitive tail is the only \'{e}tale tail.  
Then, by Proposition \ref{Pauxgroups}(i), the cover $f^{aux}$ is branched at two points with indices $e_2$ and $e_3$ such that $0 < v_p(e_2) < v_p(e_3)$, and possibly
at a third point with index $e_1$ such that $p \nmid e_1$.  In fact, $f^{aux}$ must be branched at this third point, because if it were 
branched at only two points, the Hurwitz formula would imply $e_2 = e_3$.
By the basic theory of the fundamental group in characteristic zero, the Galois group $G^{aux}$ of $f^{aux}$ 
can be generated by an element of order $e_1$ and an element of order $e_2$.  
By Proposition \ref{Pauxgroups}(ii), $G^{aux}$ has a quotient of the form $\ints/p^r \rtimes \ints/m$, for some $r$ and $m$, where $r = v_p(|G^{aux}|) \geq v_p(e_3) > v_p(e_2)$.  
But such a group cannot be generated by an element of order dividing $e_2$ and an element of order dividing $e_1$.  This is a contradiction.        
\end{proof}

\begin{remark}\label{Ralternateproof}
One can also prove Lemma \ref{Lnewtail} using a rigidity argument, along the lines of \cite[Proposition 2.8]{BP:rr}.
\end{remark}

\begin{theorem}\label{Tonepointcover}
Let $p \geq 7$ be an odd prime and $\ell$ an odd prime such that $v_p(\ell^2 - 1) = a \geq 1$.  Take $G = PSL_2(\ell)$.  Then for any algebraically closed field $k$ of characteristic $p$, 
there exists a $G$-cover of $\proj^1_k$ branched at one point whose inertia groups are dihedral of order $2p$ with higher ramification 
filtration ($\frac{3}{2}$).
\end{theorem}

\begin{proof}
If $a=1$ this is \cite[Proposition 2.8]{BP:rr}.  So assume $a \geq 2$.
By Lemmas \ref{Lexistence} and \ref{Lnewtail}, there exists a three-point $G$-cover $f: Y \to X = \proj^1_K$ whose stable reduction has a new tail 
$\ol{X}_b$ with effective ramification invariant $\sigma_b = \frac{3}{2}$.
Let $\ol{Y}_b$ be a component above $\ol{X}_b$, and let $I$ be an inertia group of $\ol{Y}_b \to \ol{X}_b$ above the unique branch point.  
We know that $I$ is of the form $\ints/p^r \rtimes \ints/m$ for some $m$ with $p \nmid m$.  
As $m_G = 2$, \cite[Lemma 4.2(iii)]{Ob:vc} shows that $\sigma_b \geq \frac{p^{r-1}}{2}$.  Since $p \geq 7$, we conclude that $r=1$.
Furthermore, since $\sigma_b \notin \ints$, the Hasse-Arf theorem shows that $I$ is not abelian.  By \cite[II, Hauptsatz 8.27]{Hu:eg}, the only nonabelian subgroup of $G$ of the form
$\ints/p \rtimes \ints/m$ is the dihedral group $D_p$ of order $2p$.  So $I \cong D_p$. 

It remains to show that the Galois group of $\ol{Y}_b \to \ol{X}_b$ is $G$, as then our cover will be given by $\ol{Y}_b \to \ol{X}_b$.  
But, since $p \geq 7$, then \cite[II, Hauptsatz 8.27]{Hu:eg} shows that the only quasi-$p$ subgroup of $G$ containing a dihedral group of order 
$2p$ is $G$ itself.  So we are done. 
\end{proof}

\begin{remark}\label{Rdirichlet}
By Dirichlet's theorem on arithmetic progressions, there are infinitely many $\ell$ satisfying the condition of Theorem \ref{Tonepointcover}
for any given $p$.
\end{remark}

\begin{corollary}\label{Cbiggergroups}
Let $k$ be an algebraically closed field of characteristic $p \geq 7$ and $G \cong PSL_2(\ell)$, where $p \, | \, |G|$. 
Suppose $I$ is either a cyclic group of order $p^r$ or a dihedral group $D_{p^r}$ of order $2p^r$, with $1 \leq r \leq v_p(|G|)$.  Then there exists a
$G$-cover $f: Y \to \proj^1_k$ branched at one point with inertia groups isomorphic to $I$.  
\end{corollary}

\begin{proof}
If $v_p(|G|) = 1$, then this is \cite[Theorem 3.6]{BP:rr}.  Suppose $v_p(|G|) > 1$.
In the case $I \cong D_p$, such a cover is given by Theorem \ref{Tonepointcover}.  If $I \cong D_{p^r}$ with $r \geq 1$, the existence of such a cover 
follows from \cite[Theorem 3.6]{Ha:ab} (taking $G=H$, $H_i = D_p$, and $H_i' = D_{p^r}$ in that theorem).  
Consider a one-point $G$-cover $f: Y \to \proj^1_k$ with inertia groups 
isomorphic to $D_{p^r}$.  Assume, without loss of
generality, that the branch point is $\infty$.  If we base change this cover by the map $\proj^1_k \to \proj^1_k$ given by $y = x^2$, then Abyhankar's 
Lemma (\cite[Lemma X.3.6]{Gr:SGA1}) shows that the inertia groups of the new cover are cyclic of order $p^r$.
\end{proof}

\begin{remark}\label{Rconductor3}
Applying Abhyankar's lemma does not change the higher ramification filtration of the $\ints/p^r$-subextension.  In particular, 
if $I \cong \ints/p$, we can take the higher ramification filtration to be $(3)$.  Alternatively, one can use \cite[Theorem 3.1]{BP:rr} to exhibit 
a cover with inertia groups $I \cong \ints/p$ and higher ramification filtration $(3)$.  Furthermore, if $\ell \equiv \pm 1 \pmod{8}$ and $I \cong \ints/p$,
then the exact same proof as \cite[Proposition 2.9]{BP:rr} (with our vanishing cycles formula (\ref{Evancycles}) substituting for the vanishing
cycles formula used there) shows that there is a cover with higher ramification filtration $(2)$.
\end{remark}

\begin{remark}\label{Rabhyankar}
This does not fully prove Abhyankar's inertia conjecture for groups $PSL_2(\ell)$ as in Theorem \ref{Tonepointcover} because we do not realize cyclic 
groups of non-$p$-power order as inertia groups (cf.\ \cite[end of \S3.1]{BP:rr}).  But note that if $I \cong \ints/p^r$ or $D_{p^r}$, 
then all subgroups of $PSL_2(\ell)$ isomorphic to $I$ are conjugate.  So the conjecture will be proven in this 
case if we can find covers with all possible isomorphism classes of cyclic inertia groups.
\end{remark}

\section{Higher ramification filtrations}\label{Shigher}
In addition to the inertia groups, one wants to understand the higher ramification filtrations of one-point covers of $\proj^1$, for the reasons mentioned in (\S\ref{Sram}).  
More generally, let $f: Y \to X$ be a one-point $G$-cover of curves over an algebraically closed field $k$ of characteristic $p$.
Assume that a $p$-Sylow subgroup of $G$ is cyclic and the inertia groups of $f$ are isomorphic to $I \cong \ints/p^r \rtimes \ints/m$,
where $p \nmid m$.  Recall that $m_I := |N_I(\ints/p^r)/Z_I(\ints/p^r)|$.  Recall also from \S\ref{Sram} that the higher ramification filtration 
for the upper numbering at any ramification point is determined by 
its sequence of positive upper jumps $(u_1, \ldots, u_r)$.  The following theorem places restrictions on this sequence:

\begin{theorem}[\cite{OP:wc}, Theorem 1.1]\label{Tadmissible}
With notations as above,  
the sequence $(u_1, \ldots, u_r)$ occurs as the upper higher ramification filtration of $I$ if and only if:
\begin{description}
\item{(a)} $u_i \in \frac{1}{m} \nats$ for $1 \leq i \leq r$;
\item{(b)} $\gcd(m, mu_1)= m/m_I$;
\item{(c)} $p \nmid mu_1$ and, for $1 < i \leq r$, either $u_i=pu_{i-1}$ or both $u_i > pu_{i-1}$ and $p \nmid mu_i$;
\item{(d)} $mu_i \equiv mu_1 \pmod{m}$ for $1 \leq i \leq r$.
\end{description}
\end{theorem}

Call a sequence $(u_1, \ldots, u_r)$ \emph{$I$-admissible} if it can occur as a sequence of positive upper jumps for $I$, based on Theorem 
\ref{Tadmissible}.  There is a partial order on $I$-admissible sequences given by
$(u_1, \ldots, u_r) \leq (u_1', \ldots, u_r')$ iff $u_i \leq u_i'$ for all $1 \leq i \leq r$.  The main result of this section is the following:

\begin{prop}\label{Pthickening}
Let $f: Y \to X$ be a one-point $G$-cover as above with inertia groups isomorphic to $I$ having upper higher ramification filtration
$\Sigma = (u_1, \ldots, u_r)$. 
Let $\Sigma' = (u_1', \ldots, u_r') \geq \Sigma$ be an $I$-admissible sequence such that $mu_1 \equiv mu_1' \pmod{m}$.  
Then there exists a one-point $G$-cover $f': Y' \to X$ with inertia groups isomorphic to $I$ having the sequence $\Sigma'$ of positive upper jumps in the 
higher ramification filtration.
\end{prop}

The key step is to show that there is a ``singular deformation'' (defined below, or see \cite[\S3]{Pr:lg}) of our original cover in a formal neighborhood of 
the branch point to a germ of a cover with the new ramification filtration $\Sigma'$.  
In the case $r=1$, the deformation result we need is given by \cite[Proposition 2.2.1]{Pr:cw}, whereas in the case $m=1$, we can use
\cite[Proposition 22]{Pr:lg}.  One can adapt the proof of \cite[Proposition 22]{Pr:lg} for general $m$, but
we will present a different proof based on the explicit equations of \cite{OP:wc}.

\subsection{Explicit equations}\label{Sexplicit}
We first write down the explicit form of any $I$-extension of the complete local ring of a point on $X$.

\begin{theorem}\label{Texplicit}
Let $k$ be an algebraically closed field of characteristic $p$ and let $I \cong \ints/p^r \rtimes \ints/m$, with $p \nmid m$.
Then any $I$-extension of $k((u))$ can be given by the following equations:
\begin{eqnarray}
x^m &=& \frac{1}{u} \label{E1}\\
y_i^p - y_i &=& f_i(y_1, \ldots, y_{i-1}, x_1, \ldots x_i),\ 1 \leq i \leq r \label{E2},
\end{eqnarray}
where the $x_i$ are polynomials in $k[x]$ with the degrees of all terms lying in a common residue class modulo $m$ and prime to $p$.  
The $f_i$ are the specific polynomials in $\FF_p[y_1, \ldots, y_{i-1}, x_1, \ldots, x_i]$
from \cite[top of p.\ 568]{OP:wc}.  The lone term of $f_i$ involving $x_i$ is $x_i$ itself.
Any element $c$ of order $m$ in $I$ satisfies $c(x)/x = \zeta$, where $\zeta$ is an $m$th root of unity.  
Furthermore, for all $i$, $c(x_i)/x_i = c(y_i)/y_i = \zeta^j$, where $\zeta^j$ is an $m_I$th root of unity.  There is an element $\sigma$
of order $p^n$ in $I$ such that $\sigma(y_i) = y_i + f_i(y_1, \ldots, y_{i-1}, 1, 0, \ldots, 0)$.  

Conversely, if $L/k((u))$ is given by the equations (\ref{E1}) and (\ref{E2}), then it can be made an $I$-Galois extension under the action of $c$ and 
$\sigma$ described above.  Its upper higher ramification filtration is $(u_1, \ldots, u_r)$, where the $u_i$ are defined inductively by 
$u_1 = \frac{\deg_x(x_1)}{m}$ and  $u_i = \max(\frac{\deg_x(x_i)}{m}, pu_{i-1})$.  
\end{theorem}

\begin{proof}
This summarizes the content of \cite[\S3, 4, 5]{OP:wc}.
\end{proof}

\subsection{Deformation}\label{Sdeformation}
We follow the notation and definitions of \cite[\S3]{Pr:lg}.  Recall that $f: Y \to X$ is a $G$-cover branched at one point with inertia groups isomorphic to 
$I \cong \ints/p^r \rtimes \ints/m$.
Let $R = k[[t]]$.  Let $U_k = \Spec k[[u]]$, and let $U_R = U_k \times_k R = \Spec k[[t,u]]$.  Let $\hat{\phi}: \hat{Z} \to U_k$ be the germ of $f$ at a ramification point $y$ of $f$.
A \emph{singular deformation} of $\hat{\phi}$ is an $I$-Galois cover $\hat{\phi}_R: \mc{Y} \to U_R$ of normal irreducible germs of $R$-curves, with 
branch locus $u = 0$, such that the normalization of the subscheme $t=0$ of $\hat{\phi}_R$ is isomorphic to $\hat{\phi}$ away from $t=u=0$.    

\begin{prop}\label{Pdeform}
In the above situation, suppose $\hat{\phi}$ has upper higher ramification filtration $\Sigma = (u_1, \ldots, u_r)$.  If $\Sigma' = (u_1', \ldots, u_r') 
\geq \Sigma$ is an $I$-admissible sequence with $mu_1 \equiv mu_1' \pmod{m}$, 
then there is a singular deformation $\hat{\phi}_R$ of $\hat{\phi}$ such that $\hat{\phi}_R \times_R k((t))$ has upper higher ramification filtration
$\Sigma'$.
\end{prop} 
 
\begin{proof}
 By Theorem \ref{Texplicit}, the morphism $\hat{\phi}$ is given by the normalization of $U_k$ in a field generated by the equations (\ref{E1}) and (\ref{E2}).  Now, consider the 
$I$-Galois extension $\hat{\phi}_R: \mc{Y} \to U_R$ given by normalizing $U_R$ in the function field generated by the equations (\ref{E1}) and (\ref{E2}), but with $x_i$ replaced by 
$x_i' := x_i + tx^{mu_i'}$ if $u_i' > pu_{i-1}'$ and $u_i' > u_i$.  The only
singular point of $\mc{Y}$ lies above $u=t=0$.  It is unramified away from $u=0$, because the right-hand sides of (\ref{E1}) and (\ref{E2}) have poles only at $u=0$.  Away from $t=0$, 
our equations for $\hat{\phi}$ give an extension in the form of Theorem \ref{Texplicit} with the $x_i$ replaced by $x_i'$.  It is easy to see that the degrees of the monomials in $x_i'$ 
(as functions of $x$) satisfy the conditions of Theorem \ref{Texplicit} when $t \ne 0$.  

We prove that, when $t \ne 0$, the $i$th upper jump $u_i''$ is equal to $u_i'$ by induction.  For $i=1$, it is clear, as both
are equal to $\deg_x(x_i')/m$.  For general $i$, assume first that $u_i' > pu_{i-1}'$.  Then $\deg_x(x_i')/m = u_i' > pu_{i-1}'$, and $u_{i-1}' = u_{i-1}''$ 
by the inductive hypothesis. 
So by Theorem \ref{Texplicit}, we have $u_i'' = u_i'$.  Now assume $u_i' = pu_{i-1}'$.  
Then $\deg_x(x_i')/m = \deg_x(x_i)/m \leq u_i \leq u_i' = pu_{i-1}' = pu_{i-1}''$.  So by Theorem \ref{Texplicit}, we have 
$u_i'' = pu_{i-1}'' = pu_{i-1}' = u_i'$.
\end{proof}

\begin{proof}[Proof of Proposition \ref{Pthickening}]
The proof is identical to that of \cite[Proposition 2.2.2]{Pr:cw}, with our Proposition \ref{Pdeform} substituting for \cite[Proposition 2.2.1]{Pr:cw}, so we only
give a sketch.  Let $\xi$ be the branch point of $f: Y \to X$. 
Construct the singular deformation at the formal neighborhood $U_k$ of $\xi$ as in Proposition \ref{Pdeform}.  
Then induce this $I$-Galois cover up to $G$, forming a disconnected
cover.  This gives the data for a \emph{relative $G$-Galois thickening problem}, which has a solution by \cite[Theorem 4]{HS:pt}.  Namely, there
is a $G$-cover $f_R: Y_R \to X_R$, where $R = k[[t]]$, where $X_R = X \times_k R$, where $f_R$ is isomorphic to the trivial deformation of $f$ 
away from $\xi \times_k R$, and where $f_R$ is our induced singular deformation above a formal neighborhood of $\xi \times_k R$.  Since the 
construction is finite in nature, we can construct a cover
$f_{R'}: Y_{R'} \to X_{R'}$ over a subring $R' \subset R$ that is of finite type over $k$ such that $f_{R'} \times_{R'} R \cong f_R$.  
Then $\Spec R'$ has infinitely many $k$-points, and for a generic $k$-point $u$, we have that $f_{R'} \times_{R'} \{u\}$ satisfies the requirements 
of the proposition. 
\end{proof}

Combining Proposition \ref{Pthickening} with the results of \S\ref{Sonepoint} we have the following:

\begin{corollary}\label{Chigherram}
Let $k$ be an algebraically closed field of characteristic $p \geq 7$, and let $G \cong PSL_2(\ell)$, such that $\ell$ is odd and $v_p(|G|) = a \geq 
1$.  Let $I \cong \ints/p^r$ or $D_{p^r}$.  Then there exists an $I$-admissible sequence $\Sigma$ for $I$ such that, for any $I$-admissible
sequence $\Sigma' \geq \Sigma$, there is a one-point cover $f: Y' \to X = \proj^1$ over $k$ with inertia groups $I$ and upper higher ramification filtration
$\Sigma'$.  If $I \cong D_p$, we can take $\Sigma = (\frac{3}{2})$.  If $I = \ints/p$, we can take $\Sigma = (3)$ (or $\Sigma = (2)$ if $\ell \equiv \pm 1
\pmod{8}$).
\end{corollary}

\begin{proof}
By Corollary \ref{Cbiggergroups}, there is a one-point $G$-cover $f: Y \to X$ with inertia groups $I$.  Let $\Sigma = (u_1, \ldots, u_r)$ be its upper
higher ramification filtration.  Suppose $\Sigma' = (u_1', \ldots, u_r') \geq \Sigma$ is any $I$-admissible sequence.  By Proposition \ref{Pthickening},
there exists a one-point $G$-cover $f': Y' \to X$ with inertia groups $I$ and upper higher ramification filtration $\Sigma'$, so long as, if $I \cong D_{p^r}$,
we have $2u_1 \cong 2u_1' \pmod{2}$.  This is true because, according to Theorem \ref{Tadmissible} (b), $2u_1$ and $2u_1'$ must both be odd.

Theorem \ref{Tonepointcover} shows that we can take $\Sigma = (\frac{3}{2})$ if $I \cong D_p$.  Remark \ref{Rconductor3}
shows that we can take $\Sigma = (3)$ if $I \cong \ints/p$, and $\Sigma = (2)$ if, in addition, $\ell \equiv \pm 1 \pmod{8}$.
\end{proof}

\begin{remark}\label{Pminimal}
One would like to find an explicit minimal choice of $\Sigma$ for general $I$.  However, the result \cite[Theorem 3.6]{Ha:ab} used in
the proof of Corollary \ref{Cbiggergroups} is not constructive as stated.  It would be interesting to look carefully through the proof of 
\cite[Theorem 3.6]{Ha:ab} to see if it can be made constructive.
\end{remark}

\backmatter

\end{document}